\documentclass[11pt]{amsart}
\usepackage{a4wide,amsfonts}
\usepackage{amsthm}
\usepackage{amsmath}
\usepackage{mathrsfs} 
\usepackage{amssymb,color}
\usepackage{amscd}
\usepackage{multirow}
\usepackage[all]{xy}

\usepackage{hyperref}

\numberwithin{equation}{section}
\usepackage{epstopdf}

\newtheorem{theorem}{Theorem}[section]
\newtheorem{lemma}[theorem]{Lemma}
\newtheorem{problem}[theorem]{Problem}

\newtheorem{corollary}[theorem]{Corollary}
\newtheorem{proposition}[theorem]{Proposition}
\newtheorem{claim}[theorem]{Claim}

\theoremstyle{definition}

\newtheorem{example}[theorem]{Example}

\newcommand{\dCF}{C_{{\downarrow}\mathsf F}}

\newcommand{\IR}{\mathbb R}
\newcommand{\IN}{\mathbb N}

\newcommand{\w}{\omega}

\newcommand{\F}{\mathcal F}
\newcommand{\Ra}{\Rightarrow}
\newcommand{\U}{\mathcal U}
\newcommand{\V}{\mathcal V}

\newcommand{\e}{\varepsilon}
\newcommand{\cN}{\mathcal N}

\begin{document}

\title[]{Lusin and Suslin properties of function  spaces}

\author{Taras Banakh}
\address{Jan Kochanowski University in Kielce and
  Ivan Franko National University in Lviv}
\curraddr{}
\email{t.o.banakh@gmail.com}
\thanks{}
\author{Leijie Wang}
\address{department of Mathematics, Shantou University, Shantou, Guangdong,
515063, PR China}
\email{leijie.wan@qq.com}
\thanks{}


\date{}

\dedicatory{}


\begin{abstract} A topological space is {\em Suslin} ({\em Lusin}) if it is a continuous (and bijective) image of a Polish space. For a Tychonoff space $X$ let $C_p(X)$, $C_k(X)$ and $\dCF(X)$ be the spaces of continuous real-valued functions on $X$, endowed with the topology of pointwise convergence, the compact-open topology, and the Fell hypograph topology, respectively. For a metrizable space $X$ we prove the equivalence of the following statements: (1) $X$ is $\sigma$-compact,  (2) $C_p(X)$ is Suslin, (3) $C_k(X)$ is Suslin, (4) $\dCF(X)$ is Suslin, (5) $C_p(X)$ is Lusin, (6) $C_k(X)$ is Lusin, (7) $\dCF(X)$ is Lusin, (8) $C_p(X)$ is $F_\sigma$-Lusin, (9) $C_k(X)$ is $F_\sigma$-Lusin, (10) $\dCF(X)$ is $C_{\delta\sigma}$-Lusin. Also we construct an example of a sequential $\aleph_0$-space $X$ with a unique non-isolated point such that the function spaces $C_p(X)$, $C_k(X)$ and  $\dCF(X)$ are not Suslin.
\end{abstract}

\maketitle

\section{Introduction}

In this paper we study the descriptive properties of the spaces $C_p(X)$, $C_k(X)$ and $\dCF(X)$ of continuous real-valued functions on a Tychonoff space $X$.

The function space $C_p(X)$ is the space $C(X)$ of continuous real-valued functions on $X$, endowed with the topology of pointwise convergence.
This topology is generated by the
subbase consisting of the sets $$
\lfloor x;r\rfloor:=\{f\in C(X):f(x)>r\}\mbox{ \ and \ } \lceil x;r\rceil:=\{f\in C(X):f(x)<r\}
$$
 where $x\in X$ and $r$ is a real number.
The function spaces $C_p(X)$ were thoroughly studied in the monographs \cite{Arh} and \cite{Tk1}, \cite{Tk2}, \cite{Tk3}.

The function space $C_k(X)$ is the space $C(X)$ endowed with the compact-open topology. This topology is generated
by the subbase consisting of the sets $$
\lfloor K;r\rfloor:=\{f\in C(X):\min f(K)>r\}\mbox{ \ and \ } \lceil K;r\rceil:=\{f\in C(X):\max f(K)<r\}
$$
 where $K$ is a nonempty compact set in $X$ and $r$ is a real number.
The function spaces $C_k(X)$ are also well-studied in General Topology  \cite[\S3.4]{Eng}, \cite{McNbook} and Functional Analysis \cite{Kakol}.

Our third  object of study is the function space $\dCF(X)$. It is the space $C(X)$ endowed with the {\em Fell hypograph topology}. This topology is generated by the
 subbase consisting of the
sets $$
\lceil K;r\rceil:=\{f\in C(X):\max f(K)<r\}\mbox{ and }
\lceil U;r\rfloor:=\{f\in C(X):\sup f(U)>r\}
$$where $K$ is a nonempty
compact subset of $X$, $U$ is a nonempty open set in $X$, and $r$ is a real number.

The study of  the function spaces $\dCF(X)$ was initiated by McCoy and Ntantu \cite{McN} and continued in \cite{WB},  \cite{Yang5}, \cite{Yang2}, \cite{Yang1},  \cite{Yang3}, \cite{Yang6}, \cite{Yang4}.
\smallskip

The function spaces $C_p(X)$ and $C_k(X)$ are Tychonoff for any topological space $X$. In contrast, $\dCF(X)$ is Tychonoff only for weakly locally compact spaces $X$.

A space $X$ is defined to be \emph{weakly locally compact} if for every
compact set $K$ in $X$
there exists an open set $U$ in $X$ such that $K\subset\bar{U}$
and  $\bar{U}$ is compact.
The  Fr\'echet-Urysohn fan (see \cite[III.1.8]{Arh}) is an example of a weakly locally
compact space, which is not locally compact.

The following characterization
of the (complete) regularity of the function spaces $\dCF(X)$ can be found in \cite[Theorem 3.7]{McN}.

\begin{theorem}[McCoy, Ntantu]\label{regular}
For a Tychonoff space $X$, the following statements are equivalent:
\begin{enumerate}
\item $\dCF(X)$ is a Tychonoff space.
\item $\dCF(X)$ is a regular space.
\item The space $X$ is weakly locally compact.
\end{enumerate}
\end{theorem}

In this paper we shall be interested in descriptive properties of the function spaces, i.e., properties that can be described in terms of Borel sets.

Let us recall that a set $A$ in a topological space $X$ is called
\begin{itemize}
\item {\em Borel} if $A$ belongs to the smallest $\sigma$-algebra of sets in $X$, containing all open subsets in $X$;
\item {\em constructible} if $A$ belongs to the smallest algebra of sets in $X$, containing all open sets in $X$;
\item {\em clopen} if it is both open and closed;
\item an {\em $F_\sigma$-set} if $A$ is a countable union of closed sets in $X$;
\item a {\em $G_\delta$-set} if $A$ is a countable intersection of open sets in $X$;
\item a {\em $C_\sigma$-set} if $A$ is a countable union of constructible  sets in $X$;
\item a {\em $C_\delta$-set}  if $A$ is a countable intersection of constructible sets in $X$;
\item an {\em $C_{\sigma\delta}$-set} (resp. {\em $F_{\sigma\delta}$-set}) if $A$ is a countable intersection of $C_\sigma$-sets (resp. $F_\sigma$-sets) in $X$;
\item an {\em $C_{\delta\sigma}$-set} (resp. {\em $G_{\delta\sigma}$-set}) if $A$ is a countable union of $C_\delta$-sets (resp. $G_\delta$-sets) in $X$.
\end{itemize}
De Morgan's laws imply that each constructible set in a topological space can be written as a finite union $(U_1\cap F_1)\cup\cdots \cup (U_n\cap F_n)$ of intersections $U_i\cap F_i$ of open and closed sets.

For any set in a topological space we have the following implications.
$$
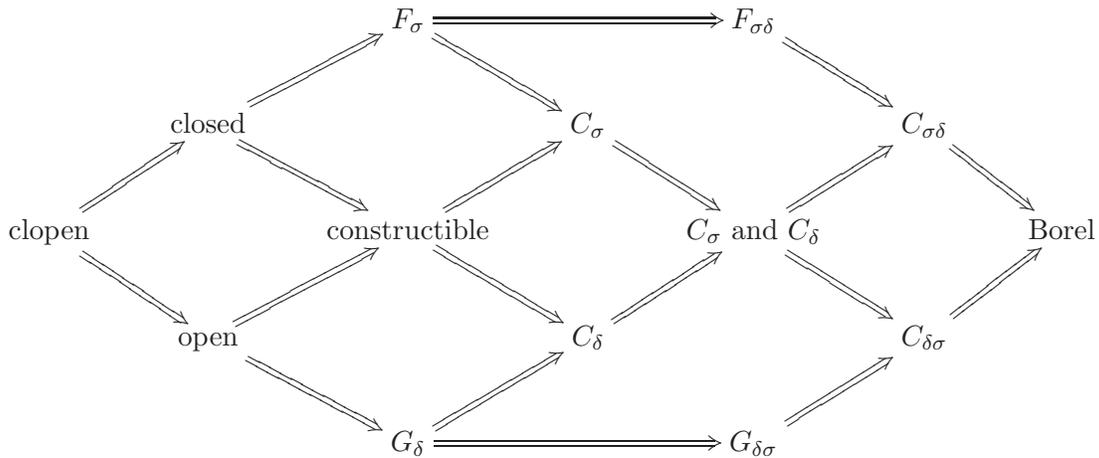
\xymatrix{
&&\mbox{$F_\sigma$}\ar@{=>}[rr]\ar@{=>}[rd]&&F_{\sigma\delta}\ar@{=>}[dr]\\
&\mbox{closed}\ar@{=>}[rd]\ar@{=>}[ur]&&\mbox{$C_\sigma$}\ar@{=>}[rd]&&C_{\sigma\delta}\ar@{=>}[dr]\\
\mbox{clopen}\ar@{=>}[ru]\ar@{=>}[rd]&&\mbox{constructible}\ar@{=>}[ru]\ar@{=>}[dr]&&\mbox{$C_\sigma$ and $C_\delta$}\ar@{=>}[ur]\ar@{=>}[rd]&&\mbox{Borel}\\
&\mbox{open}\ar@{=>}[ru]\ar@{=>}[rd]&&\mbox{$C_\delta$}\ar@{=>}[ru]&&C_{\delta\sigma}\ar@{=>}[ur]\\
&&\mbox{$G_\delta$}\ar@{=>}[ru]\ar@{=>}[rr]&&G_{\delta\sigma}\ar@{=>}[ur]
}$$
A topological space $X$ is called {\em perfect} if each open set in $X$ is of type $F_\sigma$. For example, each metrizable space is perfect. In perfect spaces the above diagram simplifies to the following form.
$$
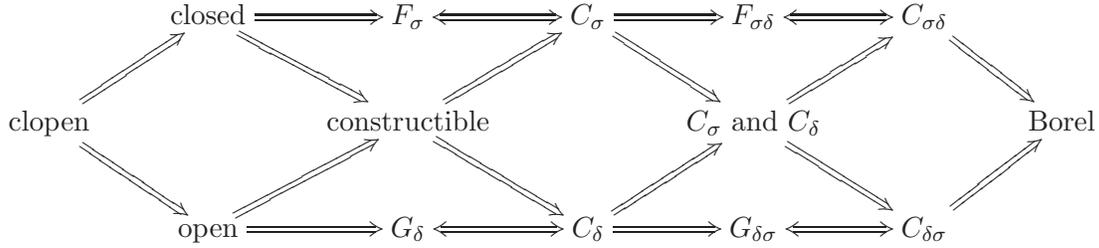
\xymatrix{
&\mbox{closed}\ar@{=>}[rd]\ar@{=>}[r]&\mbox{$F_\sigma$}\ar@{<=>}[r]&\mbox{$C_\sigma$}\ar@{=>}[rd]\ar@{=>}[r]&F_{\sigma\delta}\ar@{<=>}[r]&C_{\sigma\delta}\ar@{=>}[dr]\\
\mbox{clopen}\ar@{=>}[ru]\ar@{=>}[rd]&&\mbox{constructible}\ar@{=>}[ru]\ar@{=>}[dr]&&\mbox{$C_\sigma$ and $C_\delta$}\ar@{=>}[ur]\ar@{=>}[rd]&&\mbox{Borel}\\
&\mbox{open}\ar@{=>}[ru]\ar@{=>}[r]&\mbox{$G_\delta$}\ar@{<=>}[r]&\mbox{$C_\delta$}\ar@{=>}[ru]\ar@{=>}[r]&G_{\delta\sigma}\ar@{<=>}[r]&C_{\delta\sigma}\ar@{=>}[ur]
}$$

Let $\Gamma$ be a class of Borel sets in topological spaces. A function $f:X\to Y$ between topological spaces is called {\em $\Gamma$-measurable} if for any open set $U\subset Y$ the preimage $f^{-1}(U)$ is Borel of class $\Gamma$ in $X$. In particular, a function $f:X\to Y$ is continuous if and only if it is $G$-measurable for the class $G$ of open sets in topological spaces.

A topological space $X$ is defined to be
\begin{itemize}
\item {\em Polish} if it is homeomorphic to a separable complete metric space;
\item {\em Suslin} if it is the image of a Polish space under a continuous map;
\item {\em Lusin} if it is the image of a Polish space under a continuous bijective map;
\item {\em $\Gamma$-Lusin} for a Borel class $\Gamma$ if it is the image of a Polish space $P$ under a continuous bijective map $f:P\to X$ such that the inverse map $f^{-1}:X\to P$ is $\Gamma$-measurable.
\end{itemize}
In the role of the class $\Gamma$ we shall consider the (additive) Borel classes $G$, $F_\sigma$, $C_\sigma$, $C_{\delta\sigma}$, $G_{\delta\sigma}$, $B$ of open sets, $F_\sigma$-sets, $C_\sigma$-sets, $C_{\delta\sigma}$-sets, $G_{\delta\sigma}$-sets, Borel sets in topological spaces, respectively.

For any topological space we have the  implications:
$$\mbox{Polish $\Leftrightarrow$ $G$-Lusin $\Ra$ $F_\sigma$-Lusin $\Ra$ $C_\sigma$-Lusin $\Ra$ $C_{\delta\sigma}$-Lusin $\Ra$ Lusin $\Ra$ Suslin}.$$
For regular spaces some implications in this chain turn into equivalences (see Theorem~\ref{t:L}):
$$\mbox{$F_\sigma$-Lusin $\Leftrightarrow$ $C_\sigma$-Lusin $\Ra$ $G_{\delta\sigma}$-Lusin $\Leftrightarrow$ $C_{\delta\sigma}$-Lusin $\Ra$ $B$-Lusin $\Leftrightarrow$ Lusin $\Ra$ Suslin}.$$

By Lusin-Suslin Theorem \cite[15.1]{Ke}, a subspace  $X$ of a Polish space $P$ is Lusin if and only if $X$ is a Borel subset of $P$. By the famous result of Suslin \cite[14.2]{Ke}, each uncountable Polish space contains a Suslin subset, which is not Borel and hence is not Lusin. In the class of metrizable spaces, Lusin and Suslin spaces were defined by Bourbaki in \cite{Bourbaki}.

It is well-known \cite[4.3.26]{Eng} that each Polish (= $G$-Lusin) space $X$ is a $G_\delta$-set in any Tychonoff space containing $X$ as a subspace. In Theorem~\ref{t:Lim} we shall prove that each $F_\sigma$-Lusin space $X$ is a $C_{\sigma\delta}$-set in each Hausdorff space containing $X$ as a subspace.

Let us observe that each Suslin space has a countable network, being a continuous image of a Polish space (which has a countable base).

We recall that a family $\cN$ of subsets of a topological space $X$ is
called
\begin{itemize}
 \item a \emph{network} if for any $x\in X$ and any
  neighborhood $O_x\subset X$ of $x$, there is a set $N \in \cN$ such that
  $x\in N\subset O_x$;
\item a \emph{$k$-network} if for each open set $U\subset X$
 and compact subset $K\subset U$, there is a finite subfamily
 $\mathcal{F}\subset \cN$ such that $K \subset \bigcup\mathcal F \
 \subset U$.
\end{itemize}

A topological space $X$ is called
\begin{itemize}
\item   \emph{cosmic} if $X$ is regular and has a countable network;
\item an {\em $\aleph_0$-space} if $X$ is regular and has a
 countable $k$-network.
\end{itemize}
For any  topological space we have the following implications (see \cite[\S4]{Grue}):
$$\mbox{metrizable separable space $\Ra$ $\aleph_0$-space $\Ra$ cosmic space}.$$




Function spaces $\dCF(X)$ and $C_k(X)$ possessing countable networks were characterized in \cite[Theorem 3.7 and 4.5]{McN} and \cite{Michael} (see also \cite[\S4.1]{McNbook}).

\begin{theorem}[McCoy, Ntantu, Michael]\label{t:k} For a Tychonoff space $X$ the following conditions are equivalent:
\begin{enumerate}
\item the function space $\dCF(X)$ has a countable network;
\item the function space $C_k(X)$ is cosmic;
\item the function space $C_k(X)$ is an $\aleph_0$-space;
\item $X$ is an $\aleph_0$-space.
\end{enumerate}
\end{theorem}

The following characterization of cosmic spaces $C_p(X)$ is well-known and can be found in \cite[4.1.3]{McNbook} or \cite[I.1.3]{Arh}.

\begin{theorem}\label{t:Cp-cosmic} A Tychonoff space $X$ is cosmic  if and only if  $C_p(X)$ is cosmic.
\end{theorem}

The following fundamental result is due to Calbrix \cite{Calbrix} (see also Theorem 9.7 in \cite[p.208]{Kakol}).

\begin{theorem}[Calbrix]\label{t:Calbrix} If for a Tychonoff space $X$ the function space $C_p(X)$ is Suslin, then $X$ is $\sigma$-compact.
\end{theorem}

For an Ascoli space $X$ the Suslin property of the function space $C_k(X)$ can be characterized in terms of $\IR$-universal $\w^\w$-based uniformities.

Let us recall that a topological space $X$ is called
\begin{itemize}
\item {\em Fr\'echet-Urysohn} if for each set $A\subset X$ and point $a\in\bar A$ there exists a sequence $\{a_n\}_{n\in\w}\subset A$ that converges to $a$;
\item {\em sequential} if a subset $A\subset X$ is closed in $X$ if and only if $A$ contains the limits of all sequences $\{a_n\}_{n\in\w}\subset A$ that converge in $X$;
\item a {\em $k$-space} if a subset $A\subset X$ is closed in $X$ if and only if for any compact subset $K\subset X$ the intersection $A\cap K$ is closed in $K$;
\item a {\em $k_\IR$-space} if a function $f:X\to\IR$ is continuous if and only if for every compact subset $K\subset X$ the restriction $f{\restriction}K$ is continuous;
\item {\em Ascoli} if the canonical map $\delta:X\to C_k(C_k(X))$, $\delta:x\mapsto\delta_x$, is continuous. The canonical map assigns to each point $x$ the Dirac functional $\delta_x:C_k(X)\to\IR$, $\delta_x:f\mapsto f(x)$.
\end{itemize}
By \cite{Noble}, each Tychonoff $k_\IR$-space is Ascoli. Therefore, for any Tychonoff space we obtain the following implications:
$$\mbox{first-countable $\Ra$ Fr\'echet-Urysohn $\Ra$ sequential $\Ra$ $k$-space $\Ra$ $k_\IR$-space $\Ra$ Ascoli.}
$$
None of these implications can be reversed, see Examples 1.6.18, 1.6.19 in \cite{Eng}, \cite{Michael2}, and \cite[6.7]{BGab}.

Next, we recall some information on $\w^\w$-based uniformities. Here we consider $\w^\w$ as a partially ordered space endowed with the partial order  $\le$ defined by $\alpha\le\beta$ iff $\alpha(n)\le\beta(n)$ for all $n\in\w$.

A uniformity $\U$ on a set $X$ is called {\em $\w^\w$-based} if it has a base $(U_\alpha)_{\alpha\in\w^\w}$ such that $U_\beta\subset U_\alpha$ for any $\alpha\le\beta$ in $\w^\w$. For example, any metrizable uniformity is $\w^\w$-based.

A uniformity $\U$ on a topological space $X$ is called {\em  $\IR$-universal} if it generates the topology of $X$ and every continuous function $f:X\to\IR$ is uniformly continuous in the uniformity $\U$.

A topological space $X$ is called {\em universally $\w^\w$-based} if its universal uniformity is $\w^\w$-based. The universal uniformity of $X$ is generated by the family of all continuous pseudometrics on $X$.

\begin{theorem}\label{t:Ascoli} For an (Ascoli) Tychonoff space $X$ the function space $C_k(X)$ is Suslin if (and only if) $X$ is separable and has an $\IR$-universal $\w^\w$-based uniformity.
\end{theorem}

This theorem will be proved in Section~\ref{s:Ascoli}. Now we discuss the descriptive properties of function spaces on $\Gamma$-quotient spaces.

We say that a topological space $X$ is a {\em quotient} of a topological space $M$   if there exists a surjective quotient map $f:M\to X$. The quotient property of $f$ means that a subset $U\subset X$ is open if and only if its preimage $f^{-1}(U)$ in open in $M$.

A topological space $X$ is called {\em $\Gamma$-quotient} for a Borel class $\Gamma$ if $X$ is a quotient of some space of class $\Gamma$ in a compact metrizable space.

In particular, a {\em $G$-quotient} space is a quotient of a locally compact Polish space and an {\em $F_\sigma$-quotient} space is a quotient of a $\sigma$-compact metrizable space.

We recall that a topological space is {\em Lashnev} if it is the image of a metrizable space under a continuous closed map. It is known (see \cite[11.3]{Grue} or \cite[\S8.2]{Ban}) that separable Lashnev spaces are Fr\'echet-Urysohn $\aleph_0$-spaces.

The following characterization is proved in \cite[7.8.10 and 8.3.1]{Ban}.

\begin{proposition} A (separable Lashnev) Tychonoff space $X$ is universally $\w^\w$-based if (and only if) $X$ is $F_\sigma$-quotient.
\end{proposition}

We recall that a Tychonoff space is universally $\w^\w$-based if its universal uniformity of $X$ has an $\w^\w$-base.

Therefore, for any Tychonoff space $X$ we have the implications:
$$\mbox{$G$-quotient $\Ra$ $F_\sigma$-quotient $\Ra$ universally $\w^\w$-based},$$
where the last implication can be reversed for separable Lashnev spaces.

 The following theorem is the main result of this paper.


\begin{theorem}\label{t:main}
For a Tychonoff space $X$ consider the following statements:
\begin{enumerate}
\item $X$ is $G$-quotient;
\item $C_k(X)$ is Polish;
\item $\dCF(X)$ is $C_\sigma$-Lusin;
\smallskip
\item $X$ is $F_\sigma$-quotient;
\item $C_p(X)$, $C_k(X)$ are $F_\sigma$-Lusin and $C_{\downarrow F}(X)$ is $C_{\delta\sigma}$-Lusin;
\smallskip
\item $C_p(X)$ is Lusin and $X$ is an $\aleph_0$-space;
\item $C_k(X)$ is Lusin;
\item $C_{\downarrow F}(X)$ is Lusin;
\smallskip
\item $C_p(X)$ is Suslin and $X$ is an $\aleph_0$-space;
\item $C_k(X)$ is Suslin;
\item $C_{\downarrow F}(X)$ is Suslin;
\smallskip
\item $X$ is $\sigma$-compact.
\end{enumerate}
Then $(3)\Leftarrow(2)\Leftrightarrow(1)\Ra(4)\Ra(5)\Ra(6)\Leftrightarrow(7)\Leftrightarrow(8)\Ra(9)\Leftrightarrow(10)\Leftrightarrow(11)\Ra(12)$.
\end{theorem}

Theorem~\ref{t:main} will be proved in Section~\ref{s:proof-main} after some preliminary work done in Sections~\ref{s:LS}, \ref{s:id} and \ref{s:Fs-quotient}. This theorem implies the following characterization.

\begin{corollary}\label{c:main}
For a metrizable space $X$ the following statements are equivalent:
\begin{enumerate}
\item $X$ is $\sigma$-compact;
\smallskip
\item $C_p(X)$ is $F_\sigma$-Lusin;
\item $C_k(X)$ is $F_\sigma$-Lusin;
\item $C_{\downarrow F}(X)$ is $C_{\delta\sigma}$-Lusin;
\smallskip
\item $C_p(X)$ is Lusin;
\item $C_k(X)$ is Lusin;
\item $C_{\downarrow F}(X)$ is Lusin;
\smallskip
\item $C_p(X)$ is Suslin;
\item $C_k(X)$ is Suslin;
\item $C_{\downarrow F}(X)$ is Suslin.
\end{enumerate}
\end{corollary}

Theorem~\ref{t:main}(3), Corollary~\ref{c:main}(3) (and Theorem~\ref{regular}) suggest the following open problem.

\begin{problem} Is the function space $\dCF(X)$ \ $F_\sigma$-Lusin for any metrizable $\sigma$-compact (weakly locally compact) space $X$?
\end{problem}

The implications of  Theorems~\ref{t:main}, \ref{t:Calbrix} and \ref{t:Ascoli} are represented in the following diagram holding for any Tychonoff space $X$.
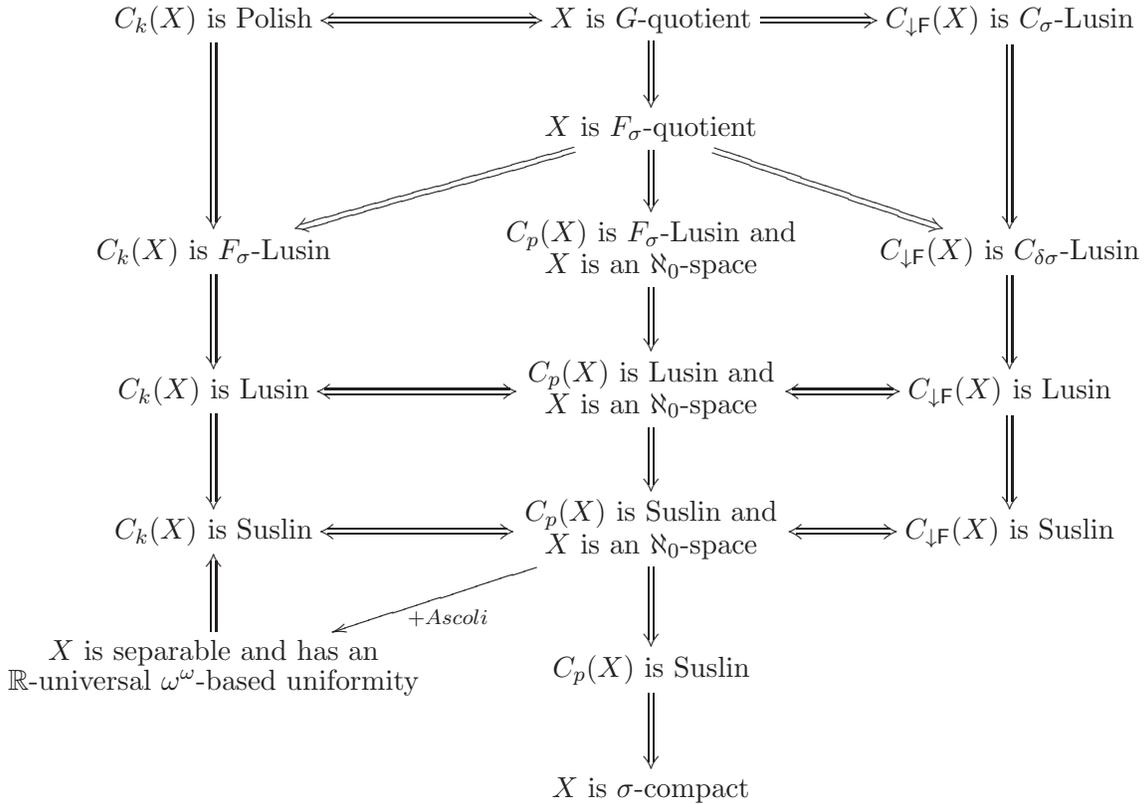
$$
\xymatrix{
\mbox{$C_k(X)$ is Polish}\ar@{<=>}[r]\ar@{=>}[dd]&\mbox{$X$ is $G$-quotient}\ar@{=>}[r]\ar@{=>}[d]&\mbox{$\dCF(X)$ is $C_\sigma$-Lusin}\ar@{=>}[dd]\\
&\mbox{$X$ is $F_\sigma$-quotient}\ar@{=>}[rd]\ar@{=>}[d]\ar@{=>}[ld]\\
\mbox{$C_k(X)$ is $F_\sigma$-Lusin}\ar@{=>}[d]&\mbox{$C_p(X)$ is $F_\sigma$-Lusin and}\atop\mbox{$X$ is an $\aleph_0$-space}\ar@{=>}[d]&\mbox{$\dCF(X)$ is $C_{\delta\sigma}$-Lusin}\ar@{=>}[d]\\
\mbox{$C_k(X)$ is Lusin}\ar@{<=>}[r]\ar@{=>}[d]&\mbox{$C_p(X)$ is Lusin and}\atop\mbox{$X$ is an $\aleph_0$-space}\ar@{=>}[d]&\mbox{$\dCF(X)$ is Lusin}\ar@{<=>}[l]\ar@{=>}[d]\\
\mbox{$C_k(X)$ is Suslin}\ar@{<=>}[r]&\mbox{$C_p(X)$ is Suslin and}\atop\mbox{$X$ is an $\aleph_0$-space}\ar@{=>}[d]\ar^{+Ascoli}[ld]&\mbox{$\dCF(X)$ is Suslin}\ar@{<=>}[l]\\
\mbox{$X$ is separable and has an}\atop\mbox{$\IR$-universal $\w^\w$-based uniformity}\ar@{=>}[u]&\mbox{$C_p(X)$ is Suslin}\ar@{=>}[d]\\
&\mbox{$X$ is $\sigma$-compact}\\
}$$

The following proposition shows that the last implication in the diagram (established by Calbrix's Theorem~\ref{t:Calbrix}) cannot be reversed even for countable Lashnev (and hence $\aleph_0$-spaces) with a unique non-isolated point. This proposition also implies that Theorem 5.7.4 in \cite{McNbook} is incorrect (that theorem claims that for any sequential $\sigma$-compact $\aleph_0$-space $X$ the function space $C_k(X)$ is Suslin).

\begin{proposition}\label{p:quot} Let $X=M/A$ be the quotient space of a metrizable space $M$ by a closed nowhere dense subset $A\subset M$. If the function space $C_k(X)$ is Suslin, then the space $M$ is $\sigma$-compact and hence $X$ is an $F_\sigma$-quotient space.
\end{proposition}

\begin{proof}  The quotient space $X=M/A$ of the metrizable space $M$ is a sequential Tychonoff space and hence Ascoli. If the function space $C_k(X)$ is Suslin, then $X$ cosmic by Theorem~\ref{t:Cp-cosmic}, and by Theorem~\ref{t:Ascoli}, the topology of $X$ is generated by some $\w^\w$-based uniformity. By Corollary 8.2.3 \cite{Ban}, the set $A$ is $\sigma$-compact. By Theorem~\ref{t:Calbrix}, the cosmic space $X$ is $\sigma$-compact and so is its open subspace $X\setminus\{A\}$, which is homeomorphic to $M\setminus A$. Then $M=A\cup(M\setminus A)$ is $\sigma$-compact and hence $X$ is $F_\sigma$-quotient.
\end{proof}

\begin{example}\label{ex}
Let $\w^{<\w}=\bigcup_{n\in\w}\w^n$ be the family of all functions $x:n\to\w$ defined on finite ordinals $n\in\w$.
Let $\w^{\le\w}=\w^{<\w}\cup\w^\w$. For any function $x\in\w^{\le \w}$ defined on an ordinal $n\le \w$ and any ordinal $k\le\w$ denote by $x{\restriction}k$ the restriction of $x$ to the ordinal $k\cap n=\min\{n,k\}$.
The set $\w^{\le\w}$ carries a natural partial order $\le$ defined by $x\le y$ iff there exists an ordinal $n\le\w$ such that $x=y{\restriction}n$.
The space $\w^{\le\w}$ is endowed with the topology $\tau$ generated by the countable base consisting of the sets ${\uparrow}x=\{y\in\w^{\le\w}:x\le y\}$ where $x\in\w^{<\w}$. It is easy to see that $(\w^{\le\w},\tau)$ is a Polish space, $\w^{<\w}$ is a dense discrete subspace in $\w^{\le\w}$ and $\w^\w$ is a closed nowhere dense subset in $\w^{\le\w}$.
Consider the quotient space $X=\omega^{\leq\omega}/{\omega^\omega}$ and observe that $X$ is a countable sequential $\aleph_0$-space with a unique non-isolated point. Since the space $\w^{\le\w}$ is not $\sigma$-compact,  the function spaces $C_p(X)$, $C_k(X)$ and $\dCF(X)$ are not Suslin according to Proposition~\ref{p:quot}. In Section~\ref{s:ex} we shall present an alternative self-contained proof of this fact.
\end{example}

Observe that the space $X=M/A$ in Proposition~\ref{p:quot} is Lashnev, i.e., the image of a metrizable space under a closed continuous map.

\begin{problem} Assume that a Tychonoff space $X$ is Lashnev and its function space $C_k(X)$ is Suslin. Is $X$ an $F_\sigma$-quotient space?
\end{problem}


\section{The Suslin property of the function spaces on Ascoli spaces}\label{s:Ascoli}

In this section we shall prove Theorem~\ref{t:Ascoli}. To prove the ``if'' part of this theorem, assume that a Tychonoff space $X$ is separable and has an $\IR$-universal $\w^\w$-based uniformity. By Theorem 7.5.1(18) of \cite{Ban}, the function space $C_k(X)$ is Suslin.
\smallskip

To prove the ``only if'' part of Theorem~\ref{t:Ascoli}, assume that the space $X$ is Ascoli and the function space $C_k(X)$ is Suslin. By Theorem~\ref{t:k}, the space $X$ is separable. By the definition of an Ascoli space, the canonical map $\delta:X\to C_k(C_k(X))$ is continuous. Since the space $C_k(X)$ is Suslin, there exists a continuous surjective map $\xi:\w^\IN\to C_k(X)$. Here $\IN=\w\setminus\{0\}$. For any $\alpha\in\w^\IN$ the subspace ${\downarrow}\alpha=\{\beta\in\w^\IN:\beta\le\alpha\}$ of $\w^\IN$ is compact and so is its continuous image $K_\alpha=\xi({\downarrow}\alpha)\subset C_k(X)$.  Observe that $K_\alpha\subset K_\beta$ for any $\alpha\le\beta$ in $\w^\w$.

For every $\alpha\in\w^\w$ consider the entourage
$$U_\alpha=\{(x,y)\in X\times X:\sup_{f\in K_{\alpha{\restriction}\IN}}|f(x)-f(y)|<2^{-\alpha(0)}\}$$
of the diagonal in $X\times X$. The continuity of the map $\delta:X\to C_k(C_k(X))$ implies that $U_\alpha$ is an open neighborhood of the diagonal in $X\times X$. It is easy to see that $(U_\alpha)_{\alpha\in\w^\w}$ is an $\w^\w$-base of some uniformity $\U$ on $X$. To see that this uniformity in $\IR$-universal, we need to show that every function $f\in C(X)$ is uniformly continuous with respect to the uniformity $\U$. Given any $\e>0$, find $\alpha\in\w^\w$ such that $2^{-\alpha(0)}<\e$ and $f=\xi(\alpha{\restriction}\IN)$. Then for any pair $(x,y)\in U_\alpha\in\U$, we get $$|f(x)-f(y)|\le\sup_{g\in K_{\alpha{\restriction}\IN}}|g(x)-g(y)|<2^{-\alpha(0)}<\e,$$
which means that $f$ is uniformly continuous.

\section{Some results on Lusin and Suslin spaces}\label{s:LS}

\begin{theorem}\label{t:Lim} Each $F_\sigma$-Lusin subspace $X$ of a Hausdorff topological space $Y$ is a $C_{\sigma\delta}$-set in $Y$. More precisely, $X=A\cap B$ for some $G_\delta$-set $A\subset Y$ and some $F_{\sigma\delta}$-set $B\subset Y$.
\end{theorem}

\begin{proof} Write $X$ as the image of a Polish space $P$ under a continuous bijective map $f:P\to X$ with $F_\sigma$-measurable inverse $f^{-1}:X\to P$. Fix a complete metric $d$ that generates the topology of the Polish space $X$.

For every $n\in\w$ fix a countable open cover $\U_n$ of $X$ by sets of $d$-diameter $<2^{-n}$. By the choice of the map $f$, for every $U\in\U_n$ the image $f(U)$ is an $F_\sigma$-set in $X$. Consequently, $f(U)=\bigcup\F_{n,U}$ for some countable family $\F_{n,U}$ of closed sets in $X$. Let $\F_n=\bigcup_{U\in\U_n}\F_{n,U}$ and $\F=\bigcup_{n\in\w}\F_n$. Therefore, $\F$ is a countable family of closed sets in $X$.
It follows that each set $F\in\F$ is equal to the intersection $X\cap\bar F$ of $X$ and the closure $\bar F$ of $F$ in the space $Y$. Observe that for any sets $F_1,\dots,F_n\in\F$ we get $\bar F_1\cap\dots\bar F_n\cap X=F_1\cap\dots\cap F_n$.

Consequently,
$$A:=Y\setminus{\textstyle\bigcup}\{\bar F_1\cap\dots\cap\bar F_n:F_1,\dots,F_n\in\F,\;F_1\cap\dots\cap F_1=\emptyset\}$$ is a $G_\delta$-set in $Y$ containing $X$.
Now consider the $F_{\sigma\delta}$-set $B=\bigcap_{n\in\w}\bigcup_{F\in\F_n}\bar F$ in $Y$ and observe that $A\cap B$ is a $C_{\sigma\delta}$-set in $Y$.
We claim that $X=A\cap B$. The inclusion $X\subset A\cap B$ is obvious.
Assuming that $X\ne A\cap B$, find a point $y\in A\cap B\setminus X$.

Then for every $n\in\w$ there exists a set $F_n\in\F_n$ such that $y\in \bar F_n$. For the set $F_n$ find an open set $U_n\in\U_n$ such that $F_n\subset f(U_n)$ and hence $f^{-1}(F_n)\subset U_n$ has $d$-diameter $<2^{-n}$.

We claim that for every $n\in\w$ the intersection $F_0\cap\dots\cap F_n$ is not empty. Assuming that this intersection is empty, we would conclude that $y\in\bar F_0\cap\dots\cap \bar F_n$ is contained in $Y\setminus A$, which contradicts the choice of $y$. Therefore, the family of closed sets $(F_n)_{n\in\w}$ is centered and so is the family $(f^{-1}(F_n))_{n\in\w}$. Since each set $f^{-1}(F_n)$ has $d$-diameter $<2^{-n}$, the completeness of the metric $d$ ensures that the intersection $\bigcap_{n\in\w}f^{-1}(F_n)$ contains a unique point $x\in X$.

By the Hausdorff property of $Y$, the point $f(x)$ has an open neighborhood $V\subset Y$ whose closure does not contain the point $y$. By the continuity of $f$ at $x$, there exists $n\in\w$ such that $f^{-1}(V)$ contain the ball $B(x;2^{-n})=\{z\in X:d(x,z)<2^{-n}\}$. Then $f^{-1}(F_n)\subset B(x;2^{-n})$ and hence $F_n=f(f^{-1}(F_n))\subset V$. Finally, $y\in \bar F_n\subset\bar V$, which contradicts the choice of $V$. This contradiction completes the proof of the equality $X=A\cap B$.
\end{proof}

 We do not know if Theorem~\ref{t:Lim} generalizes to higher Borel classes.

\begin{problem} Let $X$ be a $G_{\delta\sigma}$-Lusin subspace of a regular topological space $Y$. Is $X$ a $C_{\delta\sigma\delta}$-set (=a countable intersection of $C_{\delta\sigma}$-sets) in $Y$?
\end{problem}

Let us recall that a map $f:X\to Y$ between topological spaces is {\em Borel} if for any open set $U\subset Y$ the preimage $f^{-1}(U)$ is a Borel subset of $X$.

The following characterization of Suslin spaces  was proved in
\cite[2.5]{Banakh 1}.

\begin{theorem}\label{t:S}
A cosmic space $X$ is Suslin if and only if it is the image of a Suslin space $Z$ under a surjective Borel map $f:Z\to X$.
\end{theorem}

A similar characterization holds for Lusin spaces.

\begin{theorem}\label{t:L} For a cosmic space $X$ the following conditions are equivalent:
\begin{enumerate}
\item $X$ is Lusin;
\item $X$ is $B$-Lusin;
\item $X$ is the image of a Lusin space $Z$ under a bijective Borel map $f:Z\to X$.
\end{enumerate}
\end{theorem}

\begin{proof} $(1)\Ra(2)$: Assuming that $X$ is Lusin, find a Polish space $P$ and a continuous bijective map $f:P\to X$. The space $X$ is cosmic, being a continuous image of the Polish space. By \cite[2.9]{Grue}, the cosmic space $X$ is submetrizable and hence admits a continuous injective map $g:X\to Y$ to a Polish space $Y$. By Lusin-Suslin Theorem \cite[15.1]{Ke}, for every open set $U\subset P$ the image $g\circ f(U)$ is a Borel subset of the Polish space $Y$. The continuity of the map $g$ implies that the set $f(U)=g^{-1}(g\circ f(U))$ is Borel in $X$. This means that the space $X$ is $B$-Lusin.
\smallskip

The implication $(2)\Ra(3)$ is trivial.
\smallskip

$(3)\Ra(1)$: Assume that the cosmic space $X$ is the image of a Lusin space $L$ under a bijective Borel map $f:L\to X$. By the definition, the Lusin space $L$ is the image of a Polish space $P$ under a continuous bijective map $g:P\to L$. Then $X$ is the image of the Polish space under the Borel bijective map $h=f\circ g:P\to X$.
Being cosmic, the space $X$ has a countable network $\mathcal N$ consisting of closed subsets of $X$. Since the map $h$ is Borel, for every $N\in\mathcal N$ the preimage $h^{-1}(N)$ is a Borel subset of the Polish space $P$. By \cite[13.5]{Ke},
there exists a continuous bijective map $\xi:Z\to P$ from a Polish space $Z$ such that for every $N\in\mathcal N$ the Borel set $\xi^{-1}(h^{-1}(N))$ is open in $Z$. Consider the map $\varphi=h\circ \xi:Z\to X$ and observe that for any open set $U\subset X$ the preimage $\varphi^{-1}(U)=\bigcup\{\varphi^{-1}(N):N\in\mathcal N,\;N\subset U\}$ is open in $Z$, witnessing that the bijective map $\varphi:Z\to X$ is continuous and hence $X$ is Lusin.
\end{proof}

\section{Borel properties of the identity maps between various function spaces}\label{s:id}

It is clear that for any Tychonoff space $X$ the identity maps $C_k(X)\to C_p(X)$ and $C_k(X)\to \dCF(X)$ are continuous.

\begin{lemma}\label{l:p->k} For any $\aleph_0$-space $X$, the identity map $C_p(X)\to C_k(X)$ is $F_\sigma$-measurable.
\end{lemma}

\begin{proof}  By Theorem~\ref{t:k}, the function space $C_k(X)$ has a countable network and hence is hereditarily Lindel\"of.  Then  it suffices to find a subbase of the topology of $C_k(X)$ consisting of the sets which are Borel in the topology of the space $C_p(X)$. We claim that the standard subbase of $C_k(X)$ has this property. Given a compact set $K\subset X$ and a real number $r$, we need to check that the open sets
$$\lceil K;r\rceil:=\{f\in C_k(X):\max f(K)<r\}\mbox{ \ and \ } \lfloor K;r\rfloor:=\{f\in C_k(X):\min f(K)>r\}$$are Borel in $C_p(X)$. The compact subset $K$ of the $\aleph_0$-space has a countable network and hence is separable. Consequently, we can fix a countable dense set $\{x_m\}_{m\in\w}$ in $K$.

Now observe that the sets $$\lceil K;r\rceil=\bigcup_{n=1}^\infty\bigcap_{m\in\w}(C_p(X)\setminus \lfloor x_m;r-\tfrac1n\rfloor)\mbox{ \ and \ }\lfloor K;r\rfloor=\bigcup_{n=1}^\infty\bigcap_{m\in\w}(C_p(X)\setminus \lceil x_m;r+\tfrac1n\rceil)$$
are Borel of type $F_{\sigma}$ in $C_p(X)$.
\end{proof}

\begin{lemma}\label{l:F->p} For any cosmic space $X$, the identity map $\dCF(X)\to C_p(X)$ is $C_\sigma$-measurable.
\end{lemma}

\begin{proof}  By Theorem~\ref{t:Cp-cosmic}, the function space $C_p(X)$ has a countable network and hence is hereditarily Lindel\"of.  Then it suffices to find a subbase of the topology of $C_p(X)$ consisting of sets which are of type $C_\sigma$ in the Fell hypograph topology. We claim that the standard subbase of $C_p(X)$ has this property. Given a point $x\in X$ and a real number $r$, we need to check that the open sets
$$\lceil x;r\rceil:=\{f\in C(X):f(x)<r\}\mbox{ \ and \ } \lfloor x;r\rfloor:=\{f\in C(X):f(x)>r\}$$are Borel of type $C_\sigma$ in $\dCF(X)$. The set $\lceil x;r\rceil$ is open in $\dCF(X)$ and hence $C_\sigma$ by the definition of the Fell hypograph topology. Since $$C(X)\setminus \lfloor x;r\rfloor=\{f\in C(X):f(x)\le r\}=\bigcap_{n=1}^\infty\{f\in C(X):f(x)<r+\tfrac1n\}=\bigcap_{n=1}^\infty\lceil x;r+\tfrac1n\rceil,$$
the set $\lfloor x;r\rfloor$ is Borel of type $F_\sigma$ in $\dCF(X)$.
\end{proof}

\begin{lemma}\label{l:F->k} For any $\aleph_0$-space $X$, the identity map $\dCF(X)\to C_k(X)$ is $C_\sigma$-measurable.
\end{lemma}

\begin{proof}  By Theorem~\ref{t:k}, the function space $C_k(X)$ has a countable network and hence is hereditarily Lindel\"of.  Then it suffices to find a subbase of the topology of $C_k(X)$ consisting of sets which are of type $F_\sigma$ in the Fell hypograph topology. We claim that the standard subbase of $C_k(X)$ has this property. Given a nonempty compact set $K\subset X$ and a real number $r$, we need to check that the open sets
$$\lceil K;r\rceil:=\{f\in C(X):\max f(K)<r\}\mbox{ \ and \ } \lfloor K;r\rfloor:=\{f\in C(X):\min f(K)>r\}$$are Borel of class $C_\sigma$ in $\dCF(X)$. The set $\lceil K;r\rceil$ is open in $\dCF(X)$ by the definition of the Fell hypograph topology.

 The compact subset $K$ of the $\aleph_0$-space has a countable network $\{K_n\}_{n\in\w}$ consisting of closed (and hence compact) sets in $K$.

 Since $$
  \lfloor K;r\rfloor=\bigcup_{n\in\w}\{f\in C(X):\min f(K)\ge r+\tfrac1{2^n}\}=\bigcup_{n\in\w}(C(X)\setminus \bigcup_{m\in\w}\lceil K_m;r+\tfrac1{2^n}\rceil),
 $$
 the set $\lfloor x;r\rfloor$ is Borel of type $F_\sigma$ in $\dCF(X)$.
\end{proof}

\section{Function spaces on $F_\sigma$-quotient spaces}\label{s:Fs-quotient}

\begin{lemma}\label{l:1=>2} For any $F_\sigma$-quotient Tychonoff space $X$,  the function spaces $C_p(X)$ and  $C_k(X)$ are $F_\sigma$-Lusin and $\dCF(X)$ is $C_{\delta\sigma}$-Lusin.
\end{lemma}

\begin{proof} By the definition of an $F_\sigma$-quotient space, there exists a quotient surjective map $q:M\to X$ defined on a $\sigma$-compact metrizable space $M$. First we establish two properties of the quotient map $q$.

\begin{claim}\label{cl:a} For any sequence $\{x_n\}_{n\in\w}\subset X$ that accumulates at some point $x\in X$ there exists a sequence $\{z_k\}_{k\in\w}\subset q^{-1}(\{x_n\}_{n\in\w})$ that converges to a point $z\in q^{-1}(x)$.
\end{claim}

\begin{proof} If the set $\Omega=\{n\in\w:x_n=x\}$ is infinite, then take any point $z\in q^{-1}(x)$ and put $z_k=z$ for all $k\in\Omega$. It is clear that the sequence $(z_k)_{k\in\Omega}$ converges to $z$ and $\{q(z_k)\}_{k\in\Omega}\subset q^{-1}(\{x_n\}_{n\in\Omega})\subset q^{-1}(\{x_n\}_{n\in\w})$.

So, we assume that the set $\Omega$ is finite. Then the set $A=\{x_n:n\in\w\setminus\Omega\}$ is not closed in $X$ and by the quotient property of $q$, the preimage $q^{-1}(A)$ is not closed in $M$. Since $M$ is metrizable, there exists a sequence $\{z_k\}_{k\in\w}\subset q^{-1}(A)$, convergent to a point $z\notin q^{-1}A$. The continuity of $q$ implies that $q(z)\in \bar A\setminus A=\{x\}$.
\end{proof}

\begin{claim}\label{cl:union} For every compact set $K\subset X$ and a cover $\U$ of the set $q^{-1}(K)$ by open subsets of $M$ there exists a finite subfamily $\F\subset\U$ such that $K\subset\bigcup_{U\in\F}q(U)$.
\end{claim}

\begin{proof} Since the $\sigma$-compact space $M$ is Lindel\"of, the open cover $\U$ of the closed set $q^{-1}(K)$ contains a countable subcover $\V$. We can choose an enumeration $\{U_n\}_{n\in\w}$ of the countable family $\V$ such that $q^{-1}(K)\subset\bigcup_{n=k}^\infty U_n$ and hence $K\subset\bigcup_{n=k}^\infty q(U_n)$ for every $k\in\w$.

To finish the proof of the claim, it suffices to find $n\in\w$ such that $K\subset \bigcup_{k\le n}q(U_k)$. Assuming that no such number $n$ exists, for every $n\in\w$ we can choose a point $x_n\in K\setminus\bigcup_{k\le n}q(U_k)$. By the compactness of $K$, the sequence $(x_n)_{n\in\w}$ accumulates at some point $x\in K$. Since $K\subset\bigcup_{n\in\w}q(U_n)$, there exists a number $k\in\w$ such that $x\in q(U_k)$. Then $x\notin \{x_n\}_{n>k}$.

By Claim~\ref{cl:a}, there exists a sequence $\{z_n\}_{n\in\w}\subset q^{-1}(\{x_n\}_{n>k})$ that converges to a point $z\in q^{-1}(x)\subset q^{-1}(K)$. Since $\{U_n\}_{n>k}$ is a cover of $q^{-1}(K)$, the point $z$ belongs to some set $U_n$ with $n>k$. Since the sequence $(z_m)_{m\in\w}$ converges to the point $z\in U_n$, there exists a number $m>n$ such that $z_m\in U_n\setminus q^{-1}(\{x_{k+1},\dots,x_n\})$. Then $q(z_m)=x_i$ for some $i>n$ and hence $x_i=q(z_m)\in q(U_n)$, which contradicts the choice of $x_i$.
\end{proof}

Write the $\sigma$-compact space $M$ as the countable union $M=\bigcup_{n\in\w}M_n$ of an increasing sequence $(M_n)_{n\in\w}$ of compact subsets of $M$. Fix a countable subset $D\subset M$ such that for any $n\in\w$ the intersection $D\cap M_n$ is dense in $M_n$. Also fix a metric $d$ generating the topology of $M$. For a point $x\in M$ and a positive real number $\e$ let $B(x;\e)=\{y\in M:d(y,x)<\e\}$ be the $\e$-ball centered at $x$ and $D(x;\e):=D\cap B(x;\e)$ be the trace of the ball $B(x,\e)$ on $D$. For every $n\in\w$ let $Q_n=\{(x,y)\in M_n\times M_n:q(x)=q(y)\}$.

In the Polish space $\w^\w\times\IR^D$ consider the $G_\delta$-subset
$$
\begin{aligned}
P:=&\{(\alpha,f)\in \w^\w\times \IR^D: \forall n\in\w\;\forall x\in D\cap M_n\;\forall y\in D(x;\tfrac1{2^{\alpha(n)}})\;\;(|f(x)-f(y)|\le \tfrac1{2^n})\}\cap\\
&\{(\alpha,f)\in \w^\w\times\IR^D: \forall n\in\w\;\exists x\in D\cap M_n\;\exists y\in D(x;\tfrac2{2^{\alpha(n)}})\;\;(|f(x)-f(y)|>\tfrac1{2^n})\}\cap\\
&\{(\alpha,f)\in\w^\w\times\IR^D:\forall n\in\w\;\forall (x,y)\in Q_n\;\forall x'\in D(x;\tfrac1{2^{\alpha(n)}})\;\forall y'\in D(y;\tfrac1{2^{\alpha(n)}})\;\;\\
&\hskip100pt (|f(x')-f(y')|\le\tfrac3{2^n}\}.
\end{aligned}
$$
Observe that for every $(\alpha,f)\in P$ and every $n\in\w$ the restriction $f{\restriction}D\cap M_n$ is a uniformly continuous function, which admits a uniformly continuous extension $\bar f_n:M_n\to \IR$ to $M_n$ (by the density of $D\cap M_n$ in $M_n$). Taking into account that $D\cap M_n\subset D\cap M_{n+1}$, we conclude that $\bar f_n=\bar f_{n+1}{\restriction}M_n$, which allows us to define a function $\bar f:M\to \IR$ such that $\bar f{\restriction}M_n=f_n$ for all $n\in\w$. We claim that the function $\bar f$ is continuous. Indeed, for any $x\in M$ and any $\e>0$, we can find $n\in\w$ such that $x\in M_n$ and $\frac1{2^n}<\frac13\e$. We claim that $|\bar f(x)-\bar f(y)|<\e$ for any $y\in M$ with ${d}(x,y)<\frac1{2^{\alpha(n)}}$. Find a number $m\ge n$ such that $y\in M_m$. By the continuity of the map $\bar f{\restriction}M_m$ and the density of $D\cap M_k$ in $M_k$ for $k\in\{n,m\}$, there exist points $x'\in D\cap M_n$ and $y'\in D\cap M_m$ such that ${d}(x',y')<\frac1{2^{\alpha(n)}}$, $|\bar f(x)-\bar f(x')|<\frac13\e$ and $|\bar f(y)-\bar f(y')|<\frac13\e$.
Then
$$|\bar f(x)-\bar f(y)|\le |\bar f(x)-\bar f(x')|+|\bar f(x')-\bar f(y')|+|\bar f(y')-\bar f(y)|<\tfrac13\e+\tfrac1{2^n}+\tfrac13\e<\e.$$
Therefore the function $\bar f$ is continuous.

Next, we show that $\bar f(x)=\bar f(y)$ for any $x,y\in M$ with $q(x)=q(y)$. Assuming that $\bar f(x)\ne \bar f(y)$, we can find $n\in\w$ such that $x,y\in M_n$ and $|f(x)-f(y)|>\frac{3}{2^n}$. Then $(x,y)\in Q_n$. By the density of $D$ in $M$, there exist points $x'\in D(x;\frac1{2^{\alpha(n)}})$ and $y'\in D(y;\frac1{2^{\alpha(n)}})$ such that $|f(x')-f(y')|=|\bar f(x')-\bar f(y')|>\frac3{2^n}$.  But this contradicts the inclusion $(\alpha,f)\in P$. This contradiction shows that $\bar f=\tilde f\circ q$ for some function $\tilde f:X\to\IR$. Since the map $q$ is quotient, the function $\tilde f:X\to\IR$ is continuous.
\smallskip

So, we can consider the function $\xi:P\to C(X)$ assigning to each pair $(\alpha,f)$ the (unique) continuous function $\tilde f:X\to\IR$ such that $f=\tilde f\circ q{\restriction}D$.

\begin{claim}\label{cl:sur} The function $\xi:P\to C(X)$ is surjective.
\end{claim}

\begin{proof} Let  $\varphi:X\to\IR$ be any continuous function. For every $n\in\w$, the continuity of the function $\psi=\varphi\circ q$ at points of the compact set $M_n$ yields a number $\alpha(n)\in\w$ such that $|\psi(x)-\psi(y)|\le 2^{-n}$ for any $x\in M_n$ and $y\in M$ with ${d}(x,y)<2^{-\alpha(n)}$. We can assume that $\alpha(n)$ is the smallest possible number with this property.
Then there exists $x\in M_n$ and $y\in M$ such that ${d}(x,y)<2^{1-\alpha(n)}$ and $|\psi(x)-\psi(y)|>2^{-n}$. By the density of the sets $D\cap M_n$ in $M_n$ and $D$ in $M$, there are points $x'\in M_n\cap D$ and $y'\in D$ such that ${d}(x',y')<2^{1-\alpha(n)}$ and $|\psi(x')-\psi(y')|>2^{-n}$. It is easy to see that the pair $(\alpha,\psi{\restriction}D)$ belongs to the first two sets in the definition of the set $P$.

Let us show that $(\alpha,\psi{\restriction}D)$ also belongs to the third set. Assuming that this is not true, we can find
$n\in\w$, $(x,y)\in Q_n$ and points $x'\in D(x;2^{-\alpha(n)})$, $y'\in D(y;2^{-\alpha(n)})$ such that $|\psi(x')-\psi(y')|>\frac3{2^n}$. It follows that $x\in M_n\cap D(x';2^{-\alpha(n)})$ and $y\in M_n\cap D(y';2^{-\alpha(n)})$. By the density of the set $D\cap M_n$ in $M$, there are points $x''\in M_n\cap D(x';2^{-\alpha(n)})$ and $y''\in M_n\cap D(y';2^{-\alpha(n)})$ such that $\max\{|\psi(x)-\psi(x'')|,|\psi(y)-\psi(y'')|\}<\frac1{2^{n+1}}$. Since $x'\in D(x'';2^{-\alpha(n)})$ and $y'\in D(y'';2^{-\alpha(n)})$, the choice of $\alpha(n)$ ensures that $\max\{|\psi(x'')-\psi(x')|,|\psi(y''){\color{blue}-}\psi(y')|\}\le\frac1{2^n}$. It follows from $(x,y)\in Q_n$ that $q(x)=q(y)$ and hence $\psi(x)=\varphi(q(x))=\varphi(q(y))=\psi(y)$.
Then
$$
\begin{aligned}
\tfrac3{2^n}<&|\psi(x')-\psi(y')|\le\\
&|\psi(x')-\psi(x'')|+|\psi(x'')-\psi(x)|+|\psi(x)-\psi(y)|+|\psi(y)-\psi(y'')|+|\psi(y'')-\psi(y')|<\\
&\le\tfrac1{2^n}+\tfrac1{2^{n+1}}+0+\tfrac1{2^{n+1}}+\tfrac1{2^n}=\tfrac3{2^n},
\end{aligned}
$$
which is a contradiction finishing the proof of the inclusion $(\alpha,\psi{\restriction}D)\in P$.

Observe that for the function $f=\psi{\restriction}D$, we get $\bar f=\psi$ and $\varphi=\tilde f=\xi(\alpha,f)$, which means that the function $\xi$ is surjective.
\end{proof}

\begin{claim}\label{cl:inj} The function $\xi:P\to C(X)$ is injective.
\end{claim}

\begin{proof} Assume that $(\alpha,f),(\beta,g)\in P$ are two pairs such that $\tilde f=\tilde g$, where $\tilde f=\xi(\alpha,f)$ and $\tilde g=\xi(\beta,g)$. Then $f=\tilde f\circ q{\restriction}D=\tilde g\circ q{\restriction}D=g$.

It remains to prove that $\alpha=\beta$. Assuming that $\alpha\ne\beta$, we can find $n\in\w$ such that $\alpha(n)\ne\beta(n)$. We lose no generality assuming that $\alpha(n)<\beta(n)$. Since $(\beta,g)\in P$, there exist points $x\in D\cap M_n$ and $y\in D(x;2^{1-\beta(n)})$ such that $|g(x)-g(y)|>2^{-n}$. The strict inequality $\alpha(n)<\beta(n)$ implies  $2^{-\alpha(n)}\ge 2^{1-\beta(n)}$ and hence $y\in D(x;2^{1-\beta(n)})\subset D(x;2^{-\alpha(n)})$. Then we get points $x\in D\cap M_n$ and $y\in D(x;2^{-\alpha(n)})$ such that $|f(x)-f(y)|=|g(x)-g(y)|>2^{-n}$. But this contradicts the inclusion $(\alpha,f)\in P$.
\end{proof}

\begin{claim}\label{cl:cont} The function $\xi:P\to C_k(X)$ is continuous.
\end{claim}

\begin{proof} By \cite[11.3]{Grue}, the $F_\sigma$-quotient space $X$ is an $\aleph_0$-space. By Theorem~\ref{t:k}, the function space $C_k(X)$ is cosmic and hence hereditarily Lindel\"of. So, it suffices to show that for any nonempty compact set $K\subset X$ and any real number $r$ the sets $\xi^{-1}(\lfloor K;r\rfloor)$ and
$\xi^{-1}(\lceil K;r\rceil)$ are open in $P$.

To see that $\xi^{-1}(\lfloor K;r\rfloor)$  is open, take any pair $(\alpha,f)\in\xi^{-1}(\lfloor K;r\rfloor)$. Consider the function $\tilde f=\xi(\alpha,f)$. Since $\tilde f\in \lfloor K;r\rfloor$, there exists an open neighborhood $U\subset X$ of $K$ and a number $n\in\w$ such that $r+\tfrac2{2^n}<\inf\tilde f(U)$.

For every $z\in q^{-1}(K)$ we can find a number $k\ge n$ such that $z\in M_k$. By the density of the set $D\cap M_k$ in $M_k$, there exists a point $x\in D\cap M_k\cap q^{-1}(U)\cap B(z;2^{-\alpha(k)})$. Then  $z\in B(x;2^{-\alpha(k)})$. Therefore,
 $$q^{-1}(K)\subset \bigcup_{k=n}^\infty\{B(x;2^{-\alpha(k)}):x\in D\cap M_k\cap q^{-1}(U)\}.$$
By Claim~\ref{cl:union}, there exists $m\ge n$ and a finite subset $F\subset D\cap q^{-1}(U)$ such that $$K\subset \bigcup_{k=n}^m\bigcup_{x\in F\cap M_k}q\big(B(x;2^{-\alpha(k)})\big).$$ Consider the open neighborhood $$
W:=\bigcap_{k\le m}\bigcap_{x\in F}\{(\beta,g)\in P:\beta(k)=\alpha(k),\;g(x)>r+\tfrac1{2^n}\}$$of $(\alpha,f)$ in $P$.

We claim that $\xi(W)\subset\lfloor K;r\rfloor$.
Take any pair $(\beta,g)\in W$ and consider the function $\tilde g=\xi(\beta,g)$. Given any point $y\in K$ we should prove that $\tilde g(y)>r$. Find $k\in[n,m]$ and $x\in F\cap M_k$ such that $y\in q(B(x,2^{-\alpha(k)}))$. Then $y=q(z)$ for some $z\in B(x,2^{-\alpha(k)})$. The inclusion $(\beta,g)\in W$ ensures that $g(x)>r+\frac1{2^n}$. Let $\bar g:X\to\IR$ be the (unique) continuous function extending the function $g$. We claim that $|\bar g(z)-g(x)|\le\frac1{2^k}$. To derive a contradiction, assume that $|\bar g(z)-g(x)|>\frac1{2^k}$. By the continuity of $\bar g$ and the density of $D$ in $M$, there exists a point $t\in D$ such that $d(t,z)<2^{-\alpha(k)}-d(z,x)$ and $|\bar g(t)-\bar g(z)|<|\bar g(z)-g(x)|-\frac1{2^k}$. Then  $t\in B(x;2^{-\alpha(k)})$ and $|\bar g(t)-g(x)|>\frac1{2^k}$. The inclusion $(\beta,g)\in W$ guarantees that $\beta(k)=\alpha(k)$. Then $x\in D\cap M_k$ and $t\in D\cap B(x;2^{-\beta(k)})$ are two points with $|g(t)-g(x)|>\frac1{2^k}$, which contradicts the inclusion $(\beta,g)\in P$. This contradiction shows that $|\bar g(z)-g(x)|\le\frac1{2^k}\le\frac1{2^n}$. Then
$$\tilde g(y)=\tilde g\circ q(z)=\bar g(z)>g(x)-|\bar g(z)-g(x)|\ge r+\tfrac1{2^n}-\tfrac1{2^n}=r.$$

Therefore, $W\subset \xi^{-1}(\lfloor K;r\rfloor)$ and the set $\xi^{-1}(\lfloor K;r\rfloor)$ is open in $P$. By analogy we can prove that the set $\xi^{-1}(\lceil K;r\rceil)$  is open in $P$.
\end{proof}

\begin{claim}\label{cl:Cp->P} The function $\xi^{-1}:C_p(X)\to P$ is $F_\sigma$-measurable.
\end{claim}

\begin{proof} Since the Polish space $P$ is hereditarily Lindel\"of, it suffices to show that for any $(n,m)\in\w$, point $x\in D$ and real number $r$, the images of the subbasic open sets
$$P_{n\le m}:=\{(\alpha,f)\in P:\alpha(n)\le m\},\;\;P_{n\ge m}:=\{(\alpha,f)\in P: \alpha(n)\ge m\}$$
and
$$P_{x<r}:=\{(\alpha,f)\in P:f(x)<r\},\;\;P_{x>r}:=\{(\alpha,f)\in P:f(x)>r\}$$
under the map $\xi$ are $F_\sigma$-sets in $C_p(X)$. We shall prove that the images of these sets are open or closed (and hence $F_\sigma$) in $C_p(X)$.

Observe that the set
\begin{itemize}
\item $\xi(P_{n\le m})=\{f\in C(X):\forall x\in D\cap X_n\;\forall y\in D(x,\frac1{2^{m}})\; |f\circ q(x)-f\circ q(y)|\le \tfrac1{2^n}\}$ is closed in $C_p(X)$,
\item $\xi(P_{n\ge m})=\{f\in C(X):\exists x\in D\cap X_n\;\exists y\in D(x,\frac2{2^{m}})\; |f\circ q(x)-f\circ q(y)|>\frac1{2^n}\}$ is open in $C_p(X)$,
\item $\xi(P_{x<r})=\{f\in C(X):f\circ q(x)<r\}$ and $\xi(P_{x>r})=\{f\in C(X):f\circ q(x)>r\}$ are open in $C_p(X)$.
\end{itemize}
\end{proof}


\begin{claim}\label{cl:E} For every $x,y\in X$ and $\e>0$ the set $E=\{f\in\dCF(X):|f(x)-f(y)|>\e\}$ is of type $C_\sigma$ in $\dCF(X)$.
\end{claim}

\begin{proof} Let $\mathbb Q$ be the set of rational numbers and $Q=\{(p,q)\in\mathbb {Q}\times\mathbb Q:p+\e<q\}$. Observe that $$
\begin{aligned}
E=&\bigcup_{(p,q)\in Q}\big(\{f\in\dCF(X):f(x)<p,\;q\le f(y)\}\cup\{f\in\dCF(X):f(y)<p,\;q\le f(x)\}\big)=\\
&\bigcup_{(p,q)\in Q}(\lceil x;p\rceil\setminus \lceil y;q\rceil)\cup (\lceil y;p\rceil\setminus \lceil x;q\rceil)
\end{aligned}
$$is of type $C_\sigma$ in $\dCF(X)$.
\end{proof}

\begin{claim}\label{cl:CF->P} The function $\xi^{-1}:\dCF(X)\to P$ is $C_{\delta\sigma}$-measurable.
\end{claim}

\begin{proof} Similarly as in Claim~\ref{cl:Cp->P}, it suffices to check that for any $(n,m)\in\w$, point $x\in D$ and real number $r$, the images of the subbasic open sets $P_{n\le m},P_{n\ge m},P_{x>r}$, $P_{x<r}$ under the map $\xi$ are $C_{\delta\sigma}$-sets in $\dCF(X)$. We shall prove that the images of these sets are of type $C_\sigma$ or $C_\delta$ in $\dCF(X)$.

By Claim~\ref{cl:E}, the set
\begin{itemize}
\item $\xi(P_{n\le m})=\{f\in C(X):\forall x\in D\cap X_n\;\forall y\in D(x,\frac1{2^{m}})\;\; |f\circ q(x)-f\circ q(y)|\le \frac1{2^n}\}$ is of type $C_\delta$ in $\dCF(X)$,
\item $\xi(P_{n\ge m})=\{f\in C_p(X):\exists x\in D\cap X_n\;\exists y\in D(x,\frac2{2^{m}})\;\; |f(x)-f(y)|>\frac1{2^n}\}$ is of type $C_\sigma$ in $\dCF(X)$,
\item $\xi(P_{x<r})=\{f\in C(X):f(q(x))<r\}$ is open in $\dCF(X)$,
\item $\xi(P_{x>r})=\{f\in C(X):f(q(x))>r\}=\bigcup_{n\in\w}\{f\in C(X):f(q(x))\ge r+\frac1{2^n}\}=
\bigcup_{n\in\w}(C(X)\setminus \lceil q(x);r+\frac1{2^n}\rceil)$ is of type $F_\sigma$ in $\dCF(X)$.
\end{itemize}
\end{proof}

\begin{claim} The function space $C_p(X)$ is $F_\sigma$-Lusin.
\end{claim}

\begin{proof} By Claims~\ref{cl:sur}, \ref{cl:inj} and \ref{cl:cont}, the map $\xi:P\to C_k(X)$ is bijective and continuous. Since the identity map $C_k(X)\to C_p(X)$ is continuous, the map $\xi:P\to C_p(X)$ is continuous as the composition of two continuous maps. By Claim~\ref{cl:Cp->P}, the inverse map $\xi^{-1}:C_p(X)\to P$ is $F_\sigma$-measurable, which implies that the space $C_p(X)$ is $F_\sigma$-Lusin.
\end{proof}

\begin{claim} The function space $C_k(X)$ is $F_\sigma$-Lusin.
\end{claim}

 \begin{proof} By Claims~\ref{cl:sur}, \ref{cl:inj} and \ref{cl:cont}, the map $\xi:P\to C_k(X)$ is bijective and continuous. By Claim~\ref{cl:Cp->P} the map $\xi^{-1}:C_p(X)\to P$ is $F_\sigma$-measurable. The continuity of the identity map $C_k(X)\to C_p(X)$ implies that the map $\xi^{-1}:C_k(X)\to P$ is $F_\sigma$-measurable (as the composition of a continuous and $F_\sigma$-measurable maps). Now we see that the Polish space $P$ and the map $\xi:P\to C_k(X)$ witness that the space $C_k(X)$ is $F_\sigma$-Lusin.
\end{proof}

\begin{claim} The function space $\dCF(X)$ is $C_{\delta\sigma}$-Lusin.
\end{claim}

\begin{proof}  By Claims~\ref{cl:sur}, \ref{cl:inj} and \ref{cl:cont}, the map $\xi:P\to C_k(X)$ is bijective and continuous. Since the identity map $C_k(X)\to \dCF(X)$ is continuous, the map $\xi:P\to \dCF(X)$ is continuous (as the composition of two continuous maps). By Claim~\ref{cl:CF->P}, the inverse map $\xi^{-1}:\dCF(X)\to P$ is $C_{\delta\sigma}$-measurable, which implies that the space $C_p(X)$ is $C_{\delta\sigma}$-Lusin.
\end{proof}
\end{proof}

\section{Proof of Theorem~\ref{t:main}}\label{s:proof-main}

In this section, for a Tychonoff space $X$ we shall prove the implications
$$(3)\Leftarrow(2)\Leftrightarrow(1)\Ra(4)\Ra(6)\Ra(7)\Ra(8)\Ra(6)\Ra(7)\Ra(10)\Ra(11)\Ra(9)\Ra(10)\Ra(12)$$ of Theorem~\ref{t:main}.
\smallskip

$(2)\Ra(3)$ Assume that the function space $C_k(X)$ is Polish. The continuity of the identity map $C_k(X)\to\dCF(X)$ implies that the space $\dCF(X)$ is Lusin. Since the identity map $\dCF(X)\to C_k(X)$ is $C_\sigma$-measurable  (by Lemma~\ref{l:F->k}), the space $\dCF(X)$ is $C_\sigma$-Lusin.
\smallskip

$(2)\Ra(1)$ If the function space $C_k(X)$ is Polish, then by Corollary 5.2.5 in \cite{McNbook}, $X$ is a cosmic hemicompact $k$-space. The hemicompactness of $X$ yields an increasing sequence  $(K_n)_{n\in\w}$ of compact sets in $X$ such that each compact subset of $X$ is contained in some set $K_n$. Consider the locally compact subspace $M=\bigcup_{n\in\w}(K_n\times\{n\})$ of the product $X\times\w$ where the ordinal $\w$ is endowed with the discrete topology. Let $q:M\to X$ be the natural projection. We claim that the map $q$ is quotient. Indeed, take any subset $A\subset X$ such that the preimage $q^{-1}(A)$ is closed in $X$. Then $A\cap K_n$ is closed in $K_n$ for every $n\in\w$.  Since each compact set $K\subset X$ is contained in some $K_n$, the intersection $A\cap K=(A\cap K_n)\cap K$ is closed in $K$. Since $X$ is a $k$-space, the set $A$ is closed in $X$. Therefore, the map $q$ is quotient. Since $M$ is open in its one-point compactification, the space $X$ is $G$-quotient.
\smallskip

$(1)\Ra(2)$ Assume that $X$ is $G$-quotient and find a quotient surjective map $q:M\to X$ defined on an open subspace $M$ of a compact metrizable space. Write the locally compact space $M$ as the countable union $M=\bigcup_{n\in\w}U_n$ of an increasing sequence $(U_n)_{n\in\w}$ of open sets such that each set $U_n$ has compact closure $\overline{U_n}$, contained in $U_{n+1}$. By Claim~\ref{cl:union}, for every compact set $K\subset X$ there exists a number $n\in\w$ such that $K\subset q(U_n)\subset q(\overline{U_n})$. Now we see that the sequence of compact sets $(q(\overline{U_n}))_{n\in\w}$ witnesses that the space $X$ is hemicompact. By Theorem 11.3 \cite{Grue}, $X$ is a cosmic $k$-space, and by Corollary~5.2.5 of \cite{McNbook}, for the cosmic  hemicompact  $k$-space $X$, the function space $C_k(X)$  is Polish.
\smallskip

The implication $(1)\Ra(4)$ is trivial, $(4)\Ra(5)$ is proved in Lemma~\ref{l:1=>2} and $(5)\Ra(7)$ is trivial. The implication $(7)\Ra(8)$ follows from the continuity of the identity map $C_k(X)\to \dCF(X)$.
\smallskip

$(8)\Ra(6)$ Assume that the space $\dCF(X)$ is Lusin. Then $\dCF(X)$ has a countable network and $X$ is an $\aleph_0$-space by Theorem~\ref{t:k}. By Theorem~\ref{t:Cp-cosmic}, the function space $C_p(X)$ is cosmic. By Lemma~\ref{l:F->p}, the identity map $\dCF(X)\to C_p(X)$ is Borel and by Theorem~\ref{t:L}, the cosmic space $C_p(X)$ is Lusin.
\smallskip

$(6)\Ra(7)$ Assume that $C_p(X)$ is Lusin and $X$ is an $\aleph_0$-space. By Theorem~\ref{t:k}, the function space $C_k(X)$ is cosmic. By Lemma~\ref{l:p->k},
 the identity map $C_p(X)\to C_k(X)$ is Borel and by Theorem~\ref{t:L}, the cosmic space $C_k(X)$ is Lusin.
\smallskip

The implication $(7)\Ra(10)$ is trivial and $(10)\Ra(11)$ follows from the continuity of the identity map $C_k(X)\to \dCF(X)$.
\smallskip

$(11)\Ra(9)$ Assume that the space $\dCF(X)$ is Suslin. Then $\dCF(X)$ has a countable network and $X$ is an $\aleph_0$-space by Theorem~\ref{t:k}. By Theorem~\ref{t:Cp-cosmic}, the function space $C_p(X)$ is cosmic. By Lemma~\ref{l:F->p}, the identity map $\dCF(X)\to C_p(X)$ is Borel and by Theorem\ref{t:S}, the space $C_p(X)$ is Suslin.
\smallskip

$(9)\Ra(10)$ Assume that $C_p(X)$ is Suslin and $X$ is an $\aleph_0$-space. By Theorem~\ref{t:k}, the function space $C_k(X)$ is cosmic. By Lemma~\ref{l:p->k},
 the identity map $C_p(X)\to C_k(X)$ is Borel and by Theorem~\ref{t:S}, the  space $C_k(X)$ is Suslin.
\smallskip

$(10)\Ra(12)$ If the space $C_k(X)$ is Suslin, then so is the space $C_p(X)$ (being a continuous image of $C_k(X)$. By Calbrix's Theorem~\ref{t:Calbrix}, the space $X$ is $\sigma$-compact.

\section{Discussing Example~\ref{ex}}\label{s:ex}

In this section we prove that for the quotient space $X=\w^{\le\w}/\w^\w$ from Example~\ref{ex}, the function spaces $C_p(X)$, $C_k(X)$ and $\dCF(X)$ are not Suslin.

For a topological space $T$ denote by $T'$ the set of non-isolated points of $T$ and observe that $$C_p'(T):=\{f\in C_p(T):f(T')\subset\{0\}\}$$is a closed linear subspace of $C_p(T)$.

We recall that the discrete subspace $\w^{<\w}$ of the Polish space $\w^{\le\w}=\w^{<\w}\cup\w^\w$ carries the partial order $\le$ defined by $x\le y$ iff $x=y{\restriction}n$ for some $n\in\w$. Endowed with this partial order, the set $\w^{<\w}$ is a tree (which means that for any $x\in\w^{<\w}$ the set ${\downarrow}x=\{y\in\w^{<\w}:y\le x\}$ is finite and linearly ordered). A subtree $T$ of $\w^{<\w}$ is {\em well-founded} if it contains no infinite linearly ordered subsets.

In the function space $C_p(\w^{\le\w})$ consider the closed subspace
$$
\begin{aligned}
M_0(\w^{\le\w})=\{f\in C_p(\w^{\le\w}):&\;f(\w^\w)\subset\{0\},\;f(\w^{<\w})\subset\{0,1\},\\
&\;\forall x,y\in\w^{<\w}\;(x\le y\;\Ra f(x)\ge f(y)\}
\end{aligned}
$$consisting of non-increasing continuous functions $f:\w^{\le\w}\to\{0,1\}$ that vanish on the subspace $\w^\w$.

For any function $f\in M_0(\w^{\le\w})$ the preimage $f^{-1}(1)$ is a well-founded subtree of the tree $\w^{<\w}$. So, the space $M_0(\w^{\le\w})$ can be identified with the space $WF$ of well-founded trees on $\w$. By \cite[32.B]{Ke}, the space $WF$ is coanalytic but not analytic and so is the space $M_0(\w^{\le\w})$.
Let us recall that a subset $A$ of a Polish space $P$ is {\em analytic} (resp. {\em coanalytic}) if the space $A$ (resp. $P\setminus A$) is Suslin.
\smallskip

Consider the quotient space $X=\w^{\le\w}/\w^\w$ of the Polish space $\w^{\le\w}$ by its closed nowhere dense subspace $\w^\w$. It is clear that $X$ is a countable Tychonoff space with a unique non-isolated point $\{\w^\w\}$.
By \cite[11.3]{Grue}, $X$ is a sequential $\aleph_0$-space.

\begin{claim}\label{cl:nS} The function space $C'_p(X)=\{f\in C_p(X):f(X')\subset\{0\}\}$ is not Suslin.
\end{claim}

\begin{proof} Let $q:\w^{\le\w}\to X$ be the quotient map. It induces a continuous map $q^*:C_p(X)\to C_p(\w^{\le \w})$, $q^*:f\mapsto f\circ q$. Observe that $q^*(C_p'(X))=C'_p(\w^{\le\w})$, where $C_p'(\w^{\le\w})=\{f\in C_p(\w^{\le\w}):f(\w^\w)\subset\{0\}\}$. Assuming that the space $C_p'(X)$ is Suslin, we would conclude that its  continuous image $C'_p(\w^{\le\w})$ is Suslin.  On the other hand, $C_p'(\w^{\le\w})$ contains the non-analytic space $M_0(\w^{\le\w})$ as a closed subspace and hence $C_p'(\w^{\le\w})$ cannot be Suslin.
\end{proof}

\begin{claim} The function spaces $C_p(X),C_k(X)$ and $\dCF(X)$ are not Suslin.
\end{claim}

\begin{proof} The space $C_p(X)$ is not Suslin since it contains the closed subsapce $C'_p(X)$ which is not Suslin by Claim~\ref{cl:nS}. By Theorem~\ref{t:main}, the space $C_k(X)$ and $\dCF(X)$ are not Suslin, too.
\end{proof}

\end{document}